\documentclass[draft, reqno]{amsart}
\usepackage[all]{xy}                        %

\CompileMatrices                            % Faster

\UseTips                                    % Use

\input xypic
\usepackage[bookmarks=true]{hyperref}       % Hyperref
\usepackage{amssymb,latexsym,amsmath,amscd}
\usepackage{xspace}
\usepackage{enumerate}
\usepackage{graphicx}
\usepackage{vmargin}
\usepackage{todonotes}

\DeclareMathOperator{\red}{red}
\DeclareMathOperator{\reg}{reg}
\DeclareMathOperator{\Res}{Res}

\reversemarginpar

\theoremstyle{plain}
\newtheorem{theorem}{Theorem}[section]
\newtheorem*{theorem*}{Theorem}

\newtheorem{lemma}[theorem]{Lemma}

\newtheorem{question}[theorem]{Question}

\theoremstyle{definition}

\newtheorem{remark}[theorem]{Remark}
\newtheorem{example}[theorem]{Example}

\newcommand{\enm}[1]{\ensuremath{#1}}          %
\newcommand{\cal}[1]{\mathcal{#1}}

\newcommand{\CC}{\enm{\mathbb{C}}}

\newcommand{\PP}{\enm{\mathbb{P}}}

\newcommand{\KK}{\enm{\mathbb{K}}}

\newcommand{\Aa}{\enm{\cal{A}}}

\newcommand{\Dd}{\enm{\cal{D}}}

\newcommand{\Ii}{\enm{\cal{I}}}
\newcommand{\Jj}{\enm{\cal{J}}}

\newcommand{\Ll}{\enm{\cal{L}}}

\newcommand{\Oo}{\enm{\cal{O}}}

\newcommand{\Rr}{\enm{\cal{R}}}
\newcommand{\Ss}{\enm{\cal{S}}}

\newcommand{\Zz}{\enm{\cal{Z}}}

\renewcommand{\phi}{\varphi}
\renewcommand{\theta}{\vartheta}
\renewcommand{\epsilon}{\varepsilon}

\begin{document}

\title[secant variety]{Ranks on the boundaries of secant varieties}

\author[E. Ballico]{Edoardo Ballico}
\address[Edoardo Ballico]{Dipartimento di Matematica,  Univ. Trento, Italy}
\email{edoardo.ballico@unitn.it }
\thanks{The author was partially supported by MIUR and GNSAGA of INdAM (Italy).}

\begin{abstract} In many cases (e.g. for many Segre or Segre embeddings of multiprojective spaces) we prove that a hypersurface of the $b$-secant variety
of $X\subset \PP^r$ has $X$-rank $>b$. We prove it proving that the $X$-rank of a general point of the join of $b-2$ copies of $X$ and the tangential variety
of $X$ is $>b$. \end{abstract}

\maketitle
\section{Introduction}

Let $X\subset \PP^r$ be an integral and non-degenerate variety defined over an algebraically closed field. For any $q\in X$ the $X$-rank $r_X(q)$ of $X$
is the minimal cardinality of a set $S\subset X$ such that $q\in \langle S\rangle$, where $\langle \ \rangle$ denote the linear span.
For any $q\in \PP^r$ let $\Ss (X,q)$ denote the set of all finite subsets $S\subset X$ such that $q\in \langle S\rangle$ and $\sharp (S)=r_X(q)$. For any integer
$s>0$ let $\sigma _s(X)\subseteq \PP^r$ be the $s$-secant variety of $X$, i.e. the closure of the union of all linear space
$\langle S\rangle$ with $S\subset X$ and $\sharp (S)=s$. See \cite{l} for many applications of $X$-ranks (e.g. the tensor rank) and secant varieties (a.k.a. the border rank). The algebraic set $\sigma _s(X)$ is an integral variety
of dimension $\le s(1+ \dim X) -1$ and $\sigma _s(X)$ is said to be non-defective if it has dimension $\min \{r,s(1+\dim X)-1\}$. Every secant variety of a curve
is non-defective (\cite[Corollary 4]{a}). Let $\tau (X)\subseteq \PP^r$ be the tangential variety of $X$, i.e. the closure in $\PP^r$ of the union of all tangent spaces
$T_pX$, $p\in X_{\reg}$. The algebraic set $\tau (X)$ is an integral variety of dimension $\le 2(\dim X)$ and $\tau (X)\subseteq \sigma _2(X)$. For any integer $b\ge 2$ let $\tau (X,b)$ denote the join of one copy of $\tau (X)$ and $b-2$ copies of $X$. If $X$ is a curve, then $\dim \tau (X,b) = \min \{r,s(1+\dim X)-2\}$ (use $b-2$ times \cite[part 2) of Proposition 1.3]{a}
and that $\dim \tau (X)=2$) and hence $\tau (X,b)$ is a non-empty codimension $1$ subset of $\sigma _b(X)$ if $X$ is a curve and $r> 2b$. For a variety $X$ of arbitrary dimensional
usually $\tau (X,b)$ is a hypersurface of $\sigma _b(X)$, but this is not always true. For instance, if $\sigma _2(X)$ has not the expected dimension one expects that
$\tau (X,b) = \sigma _b(X)$ and this is the case if $X$ is smooth (\cite[Corollary 4]{fh}).
A general $q\in \tau (X)$ has $r_X(q) =2$ (and hence for any $b \ge 2$ a general $o\in \tau (X,b)$ has $X$-rank $\le b$)
if a general tangent line to $X_{\reg}$ meets $X$ at another point of $X$, i.e. if $X$ is \emph{tangentially degenerate}
in the sense of \cite{k}. It is easy to check that $X$ is tangentially degenerate if and only if the curve $X\cap M\subset M$ is tangentially degenerate, where $M$ is a general
codimension $n-1$ linear subspace of $X$. H. Kaji proved that in characteristic zero a smooth curve in $\PP^m$, $t\ge 3$,
is not tangentially degenerate (\cite[Theorem 3.1]{k}) and this is true also if the normalization map of $X\cap M$ is unramified (\cite[Remark 3.8]{k}) or if $X\cap M$ has only toric
singularities (\cite{bp}). See \cite{k1} for the state of the art (at that time) on tangentially degenerate curves and a list to the examples known in positive characteristic.

In \cite{b} we raised the following question and gave a positive answer (in characteristic zero) when $X$ is a curve.

\begin{question}\label{z1}
Assume $b\ge 2$, $r\ge b(1+\dim X)-2$, and that $\sigma _s(X)$ has the expected dimension. Is $r_X(q)>b$ for a non-empty subset of $\sigma _b(X)$ of codimension
$1$ in $\sigma _b(X)$? Is $r_X(q)>b$ for a general point of $\tau (X,b)$?
\end{question}

Our aim is to refine this question for $n:= \dim X >1$ and get (in some cases) a positive answer. Take a general $q\in
\tau(X,b)$. There is $o\in X_{\reg}$, a degree $2$ connected zero-dimensional scheme $v$ with $v_{\red} = \{o\}$ and $p_1,\dots ,p_{b-2}\in X$ such that $p_i\ne p_j$ for all $i\ne j$, $p_i\ne o$ for all $i$ and $q\in
\langle v\cup
\{p_1,\dots ,p_{b-2}\}\rangle$. For a general $q\in \tau (x,b)$ the set $(o,p_1,\dots ,p_{b-2})$ is general in $X^{b-1}$ and $v$
is a general tangent vector to $X$ at $o$. Let $\Zz (X,b)$ be the set of all degree $b$ schemes $v\cup \{p_1,\dots ,p_{b-2}\}$
with
$p_i\ne p_j$ for all $i\ne j$ and $o:= v_{\red} \in X_{\reg}\setminus \{p_1,\dots ,p_{b-2}\}$. Let $\tau (X,b)'$ be the union of all
$q\in \tau (X,b)$ such that there is $Z\in \Zz (X,b)$ with $q\in \langle Z\rangle$. For any $q\in \tau (X,b)'$ let $\Zz
(X,b,q)$ be the set of all $Z\in \Zz (X,b)$ such that $q\in \langle Z\rangle$. 

\quad (i) Is $\dim \tau (X,b) = b(n+1)-2 = \dim \sigma _b(X)-1$?

\quad (ii) Is $\sharp (\Zz (X,b,q)) =1$ for a general $q\in \tau (X,b)'$?

\quad (iii) Is $r_X(q) >b$ for a general $q\in \tau (X,b)$? 

If (i) and (iii) are true, then the set of all $q\in \sigma _b(X)$ with $r_X(q)>b$ has dimension $b(n+1)-2$ (i.e. over $\CC$ it has Hausdorff dimension $2b(n+1)-4$). To get a positive answer the first part of  Question \ref{z1} for $X$ and $b$ it is not necessary to prove that (i) and (iii) hold and probably (ii) never
will be used to prove (i) and (iii), but (ii) is a nice question, similar to ask if $\sharp (\Ss (X,o)) =1$ for a general $o\in \sigma
_b(X)$ (this is called the identifiability of $\sigma _b(X)$). The way we prove  (iii) in the next theorem we get with a very
similar proof also (ii), while (i) comes for free. 

We prove the following result.

\begin{theorem}\label{i2}Take $b\ge 2$. 
Let $X\subset \PP^r$, be a an integral and non-degenerate variety, which is non-singular in codimension $1$. Set $n:= \dim X$.
Assume $\Oo _X(1) =\Ll\otimes \Rr$ and the existence of base point free linear spaces $V\subseteq H^0(\Ll)$, $W\subseteq
H^0(\Rr)$ such that $v:= \dim V\ge n+b+2$, $w:= \dim W \ge n+b+2$, the morphisms $u_V: X\to \PP^{v-1}$ and $u_W:
X\to
\PP^{w-1}$, are birational onto their images, that the closures of their images $X_V$ and $X_W$ have singular locus of
dimension
$\le n-1$, and that $\dim \sigma _2(X_V) =2n+1$.  Assume that the image of the multiplication map
$V\otimes W
\to H^0(\Oo _X(1))$ is contained in the image of the restriction map $H^0(\Oo _{\PP^r}(1)) \to H^0(\Oo _X(1))$ and it induces
an embedding.  Then
$\dim
\tau (X,b) =b(n+1)-2 = \dim \sigma _b(X) -1$ and a general $q\in \tau (X,b)$ has $r_X(q) >b$ and $\sharp (\Zz (X,b,q))=1$. 
\end{theorem}
If $u_v$ and $u_W$ are embeddings the assumptions on the singularities of $X_V$ and $X_W$ are satisfied if and only if $X$ is non-singular in codimension $1$.

We apply Theorem \ref{i2} to the case of certain Segre-Veronese embeddings of multiprojective spaces (see Example \ref{i1}),
but since we assumed that both $\Ll$ and $\Rr$ are birationally very ample, we cannot useTheorem \ref{i2} for the most important case:
tensors, i.e. the Segre embedding of a multiprojective space.  For tensors we prove the following result.

\begin{theorem}\label{i3}
Let $X\subset \PP^r$, $r+1=\prod _{i=1}^{s} (n_i+1)$,  be the Segre embedding of the multiprojective space
$\PP^{n_1}\times
\cdots \times \PP^{n_s}$. Fix an integer $b\ge 2$ and assume the existence of a decomposition
$E\sqcup F = \{1,\dots ,s\}$ such that $\prod _{i\in E} (n_i+1) > b+3+\sum _{i\in E} n_i$ and $\prod _{i\in E} (n_i+1) > b+3+\sum _{i\in F} n_i$.
Then $\dim \sigma _b(X) =b(n+1)-1$, $\dim \tau (X,b) =b(n+1)-2$ and $r_X(q) >b$ for a general $q\in \tau (X,b)$.\end{theorem}

The assumptions of Theorem
\ref{i3} imply $\sharp (E)\ge 2$ and $\sharp (F)\ge 2$ and hence they exclude the case
$s=2, 3$. The exclusion of the case $s=2$ is not a fault of our too restrictive assumptions. If
$s=2$ every
$q\in
\tau (X)\setminus X$ has
$X$-rank
$2$ (\cite{bb1},
\cite[Proposition 1.1]{bl}) and hence a general $q\in \tau (X,b)$ has rank at most $b$. The paper \cite{bbcg}  contains $3$ results related to Theorem \ref{i3} (\cite[Theorems 3.1, 4.6 and 4.10]{bbcg}), but none of them covers
Theorem \ref{i3}. 

For a better description of the $X$-ranks of $\sigma _3(X)$ for $s=3$ see \cite{bl}. In this case $\tau (X)$
is not contained in the singular locus of $\sigma _2(X)$ (\cite[Theorem 1.3]{bl}. We expect that the same holds for $\tau
(X,b)$ for certain very positively embedded $X$. For the case $b=2$, see \cite{u}.

 We work over an algebraically closed field $\KK$ with characteristic zero.

\section{Proof of Theorem \ref{i2}}

For any integer $b>0$ let $\Aa (X,b)$ denote the set of all subsets of $X$ with cardinality $b$. For any zero-dimensional
scheme $Z\subset X$ and any effective Cartier divisor $D$ of $X$ the residual scheme of $Z$ with respect to $D$ is the closed
subscheme of $X$ with $\Ii _Z:\Ii _D$ as its ideal sheaf. We have $\Res _D(Z)\subseteq Z$ and
$\deg (Z) =\deg (Z\cap D)+\deg (\Res_D(Z))$. For any line bundle $\Ll$ on $X$ we have an exact sequence (the residual sequence
of $\Ii _Z\otimes \Ll$ with respect to $D$):
\begin{equation}\label{eqt1}
0 \to \Ii _{\Res _D(Z)}\otimes \Ll (-D)\to \Ii _Z\otimes \Ll \to \Ii _{Z\cap D,D}\otimes \Ll _{|D}\to 0
\end{equation}
For any $\Ll \in \mathrm{Pic}(X)$, any linear subspace $V\subseteq H^0(X,\Ll)$ and any zero-dimensional scheme $Z\subset X$ set
$V(-Z):= V\cap H^0(X,\Ii _Z\otimes \Ll)$.

For any integral variety $M$ and any $o\in M_{\reg}$ let $(2o, M)$ be the first infinitesimal neighborhood of $o$ in $M$, i.e. the closed subscheme of $M$ with
$(\Ii _{o,M})^2$ as its ideal sheaf.

\begin{lemma}\label{zz1}
Let $X\subsetneq \PP^r$, be an integral and non-degenerate variety, which is
scheme-theoretically
cut out by quadrics. Then $X$ is not tangentially degenerate.
\end{lemma}

\begin{proof}
Take a general $q\in X_{\reg}$ and a general line $L\subset \PP^r$ tangent to $X_{\reg}$ at $q$ and assume
that $(L\cap X)_{\red}$ contains a point $o\ne q$. Since the connected component of $L\cap X$ containing $q$ contains the
divisor
$2q$ of $L$ and $X$ is scheme-theoretically cut out by quadrics, we have $L\subset X$. Since $L$ is general, we get
$\tau (X)\subseteq X$ and so $\tau (X)=X$. Let $M\subset \PP^r$ be a general linear space with codimension $n-1$.
The scheme $X\cap M$ is an integral curve spanning $M$ and we get $\tau (X\cap M) =X\cap M$, contradicting the assumption
$X\subsetneq \PP^r$.
\end{proof}

\begin{remark}\label{zz0}
The homogeneous ideal of a Segre-Veronese variety $X\subset \PP^r$ is generated by the $2\times 2$ minors of flattenings
(\cite[Theorem 6.10.6.5]{l}) and in particular (unless $X=\PP^r$) it is not tangentially degenerate by Lemma \ref{zz1}. Just
to know that $X$ is scheme-theoretically cut out by quadrics (to be able to apply Lemma \ref{zz1}) is easier, since this is
easily seen to be true if it is true for the Segre embedding of $X$.
\end{remark}

\begin{lemma}\label{a1}
Let $X\subset \PP^r$, $r\ge 3+n$,  be an integral and non-degenerate $n$-dimensional variety, which is non-singular in codimension 1. Let $L\subset \PP^3$ be a
general tangent line to $X_{\reg}$. Let $\ell _L: \PP^r\setminus L\to \PP^{r-2}$ denote the linear projection from $L$. Then $\ell _{L|X\setminus X\cap L}$ is birational onto its
image.
\end{lemma}

\begin{proof}
Since we are in characteristic zero, it is sufficient to prove that $\ell _{L|X\setminus X\cap L}$ is generically injective, i.e. that for a general $q\in X$ the plane $\langle L\cup \{q\}\rangle$ intersects $X$ only in $q$ and the set $(X\cap L)_{\red}$. If $n=1$, then this is true by \cite[Lemma 2.5]{b}. Now assume $n>1$ and that the lemma is true for varieties of dimension $<n$.
Since
$n>1$, for a general hyperplane $M\subset \PP^n$ the scheme $X\cap M$ is an integral variety non-singular in codimension $1$ and spanning $M$. Since $\dim X\cap
M>1$, some tangent line of $X_{\reg}$ is contained in $M$ and it is tangent to $(X\cap M)_{\reg}$. Since $L$ is a general
tangent line
of $X_{\reg}$, we get that for a general hyperplane $H\supset L$ the scheme $X\cap H$ is integral and spans $M$.
Since $L$ is a general tangent line, the set $X\cap L$ is finite. Take $p\in X_{\reg}\cap L$ such that $L\subset T_pX$. Since
$\dim T_pX >1$ a general $H\supset L$ does not contain $T_pX$, i.e. $X\cap H$ is smooth at $p$. We move $L$ among the tangent lines
of $(X\cap H)_{\reg}$ and apply the inductive assumption to $X\cap H$. We get that for a general $q\in X\cap H$ the
plane
$\langle L\cup
\{q\}\rangle$ intersects $X\cap H$ (and hence $X$) only in $q$ and the set $(X\cap L)_{\red}$.
Moving
$H$ among the hyperplanes containing $L$ we get the lemma.
\end{proof}

\begin{lemma}\label{a2}
Fix an integer $b\ge 2$. Let $X\subset \PP^r$, $r\ge 4+n$,  be an integral and non-degenerate $n$-dimensional variety  which is non-singular in codimension $1$ and
take a general $Z\in \Zz (X,b)$. Write $Z = v\sqcup \{p_1,\dots ,p_{b-2}\}$ with $\deg (v)=2$ and
$v$ connected. Set $L:= \langle v\rangle$ and $M:=\langle Z\rangle$. Then
$\dim M
 =b-1$ and
$X\cap
M = \{p_1,\dots ,p_{b-2}\}\cup (X\cap L)$ (as schemes).
\end{lemma}

\begin{proof}
If $b=2$, then $L=M$ and the lemma is trivial. Now assume $b>2$. We have $\dim M =b-1$, because $p_1,\dots ,p_{b-2}$ are
general, $X$ is non-degenerate and $n+1\le r$. Let
$\ell _L:
\PP^r\setminus L\to
\PP^{r-2}$ denote the linear projection from $L$. Let $Y\subset \PP^{r-2}$ be the closure of $\ell _L(X\setminus X\cap L)$. By
Lemma
\ref{a1}
$\ell _L$ sends $X\setminus L\cap X$ birational into $Y$ and $\dim Y =n$. Since $Z$ is general, we have $p_i\notin L$
for all $i$ and hence the points $q_i:= \ell _L(p_i)$ are well-defined. For a general $Z$ the $b$-tuple $(q_1,\dots ,q_{b-2})$ is general
in $Y^{b-2}$. Hence $N:= \langle \{q_1,\dots ,q_{b-2}\}\rangle$ has dimension $b-3$. Since we are in characteristic zero, the
trisecant lemma (also known as the uniform position principle) (\cite[p. 109]{acgh})  implies that $N\cap Y =\{q_1,\dots ,q_{b-2}\}$  (as schemes). Since $p_1,\dots ,p_{b-2}$ are
general and $\ell _L$ is birational onto its image, we get the lemma.
\end{proof}

\begin{remark}\label{zzz2}Let $X\in \PP^r$ be an integral and non-degenerate variety. Set $n:= \dim X$. Let $\tau (X)\subseteq
\PP^r$ be the
tangential variety of $X$.
 In characteristic zero if $\tau (X)\ne \PP^r$ we have $X\subseteq \mathrm{Sing}(\tau (X))$. For a general $x\in \tau
(X)$ there is $o\in X_{\reg}$ and a line $L\subseteq T_oX$ with $x\in L\setminus \{o\}$. The tangent space of $\tau (X)$ is
constant at all points of $\tau (X)_{\reg}\cap L$. $L$ is uniquely determined by a degree $2$ zero-dimensional scheme $v\subset M$ such that $v_{\red} =
\{o\}$. Let $Z(o,v)$ denote the following zero-dimensional scheme of $X$ (and hence of $\PP^r$) with
$Z(o,v)_{\red} =\{o\}$ and $\deg (Z(o,v))=2n+1$. It is sufficient to define the ideal $\Jj$ of $Z(o,v)$ in the local ring $\Oo
_{X,o}$.  Since $\Oo _{X,o}$ is an
$n$-dimensional regular local ring, there is a regular system of parameters $x_1,\dots ,x_r$ such that $x_1^2,x_2,\dots ,x_n$
generate the ideal sheaf of $v$ in $\PP^r$. Take as $\Jj$ the ideal generated by all $x_ix_jx_k$, $i,j,k\in \{1,\dots, ,n\}$
and all $x_1x_i$, $i=1,\dots ,x_n$. Now we check that this definition depends only on $X$, $o$ and $v$,
but not on the choice of $x_1,\dots ,x_r$. Let $\mu$ be the maximal ideal of $\Oo _{X,o}$. Take another regular system of
parameters
$y_1,\dots ,y_n$ of
$\Oo _{X,o}$ with $y_1^2,y_2,\dots, y_n$ generating the ideal sheaf of $v$ in $X$. Since $\Oo _{X,o}$ is regular, the
completion
$\hat{\Oo}_{X,o}$ of
$\Oo _{X,o}$ with respect to its maximal ideal is isomorphic to
$\KK[[x_1,\dots ,x_n]]$. In $\KK[[x_1,\dots ,x_n]]$ we have $y_i=L_i+M_i$, with 
$M_i$ a power series with no constant and no linear term, $L_1,\dots ,L_n$ linear forms in $x_1,\dots ,x_n$ with invertible Jacobian with respect to $x_1,\dots
,x_n$ and there is a non-zero constant $c$ such that $x_1-cy_1\in (x_2,\dots ,x_n)+\mu ^2$. Thus $y_1,\dots ,y_n$ gives the
same ideal. We have $T_x\tau (X)\supset Z(o,v)$. Now assume that the scheme $Z(o,v)$ is linearly independent in $\PP^r$, i.e.
that $\dim \langle Z(o,v)\rangle =2n$. Since $\dim T_o\tau (X)=2n$ and $T_o\tau (X)\supset Z(o,v)$, we get $\langle Z(o,v)\rangle = T_o\tau (X)$.
\end{remark}

\begin{proof}[Proof of Theorem \ref{i2}:]
Taking a linear projection we reduce to prove the theorem when the map
$V\otimes W\to H^0(\Oo _{\PP^r}(1))$ is surjective. Let $I_V$ (resp. $I_W$) be proper closed subschemes of
$X$ such that
$u_V$ (resp.
$u_W$) is an embedding over
$X\setminus I_V$ (resp. $X\setminus I_W$) and $u_V^{-1}(u_V(X\setminus I_V)) = X\setminus I_V$ (resp. $u_W^{-1}(u_W(X\setminus
I_W)) = X\setminus I_W$).

\quad (a) In this step we prove that $\dim \sigma _b(X) =b(n+1)-1$. 
Fix a general $S\subset X$ with $\sharp (S)=b$. Set $Z:= \cup _{o\in S} (2o,X)$. By \cite[Corollary 1.10]{a} it is sufficient to prove that $h^1(\Ii _Z(1)) =0$. We use induction on the integer $b$, starting the induction here with the obvious case $b=1$. Fix $o\in S$ and set $S':= S\setminus \{o\}$ and $B:=  \cup _{o\in S'} (2o,X)$. By the inductive assumption
we may assume $h^1(\Ii _B(1)) =0$. Thus it is sufficient to prove that $(2o,X)$ gives $n+1$ independent conditions to $H^0(\Ii _B(1))$.
Since we may take $o$ general after fixing $S'$, $o$ is not in the base locus of $H^0(\Ii _B(1))$. Take $N\subseteq T_oX$ with
$0 \le \dim N \le n$ and maximal giving independent conditions to $H^0(\Ii _{B\cup \{o\}}(1))$ and let $N'\subseteq (2o,X)$
the corresponding zero-dimensional scheme with $\deg (N') =1 +\dim N$. Assume $N'\ne (2o,X)$ and fix $N''\subseteq (2o,X)$
with $\deg (N'') = \deg (N')+1$. To get a contradiction it is sufficient to prove that $H^0(\Ii _{N''}(1))\subsetneq H^0(\Ii
_{N'}(1))$. Since $S$ is general, we may assume $S\cap I_V = S\cap I_W =\emptyset$. Since $u_V$ and $u_W$ are embedding at
$o$, we have $\dim V(-N'') =\dim V(-N')-1$. Since we may take $S'$ general after fixing $o$, we have  $\dim V(-N'') =\dim
V(-N')-1$. Take $f\in V(-N'-S')\setminus V(-N''-S')$. Since we may take $o$ general after fixing $S'$, we have
$W(-S) \ne W(-S')$. Take $g\in W(-S') \setminus W(-S)$. The image of $f\otimes g$ shows that $H^0(\Ii _{N''}(1))\subsetneq H^0(\Ii _{N'}(1))$.

\quad (b) In this step we prove that $\dim \tau (X,b) =b(n+1)-2$. Fix a general $(o,o_1,\dots ,o_{b-2})\in X_{\reg}^{b-1}$ and a general tangent vector $v$ to $X$ at $o$.
Let $Z'$ be the degree $2n+1$ scheme associated to $o$ and $v$ as in Remark \ref{zzz2}. Set $Z'':= (2o_1,X)\cup \cdots \cup
(2o_{b-2},X)$ and
$Z:= Z'\cup Z''$. Since $\tau (X,b)$ is the join of $\tau (X)$ and $b-2$ copies of $X$, by Terracini lemma it is sufficient to
prove that $h^1(\Ii _Z(1)) =0$. For a general $(o,o_1,\dots ,o_{b-2})$ we may assume $o\notin (I_V\cup I_W)$. Since $u_V$ and
$u_W$ are embeddings at $o$, we have $\dim V(-(2o,X)) =v-n-1$ and $\dim W(-(2o,X)) =w-n-1$. Using $V\otimes W$ we see that
$H^0(\Oo _{\PP^r}(1))$ separates
the $2$-jets of $X$ at $o$ and in particular $h^1(\Ii _{Z'}(1)) =0$, concluding the proof of the case $b=2$. We
proved also that $(3o,X)$ is linearly independent in $\PP^r$, where $(3o,X)$ is the closed subscheme of $X$ with $(\Ii _{o,X})^3$ as
its ideal sheaf. Now assume
$b>2$ and that the last assertion is true for the integer $b-1$, i.e. assume that the zero-dimensional scheme $E:= (3o,X)\cup(2o_1,X)\cup
\cdots \cup (2o_{b-3},X)$ is linearly independent. To prove that $\dim \tau (X,b) =b(n+1)-2$ it is sufficient to prove that
$(3o,X)\cup(2o_1,X)\cup
\cdots \cup (2o_{b-2},X)$ is linearly independent. Fix an integer $a\in \{0,\dots ,n\}$ and schemes $A_1\subset A_2\subseteq
(2o_{b-2},X)$ with $\deg (A_1) =\deg (A_2)-1 =a$. By induction on $a$ it is sufficient to prove that $H^0(\PP^{v-1},\Ii _{E\cup A_2}(1))
\subsetneq H^0(\PP^{v-1},\Ii _{E\cup A_1}(1))$. Since $\dim  \sigma _2(X_V)=2n+1$, Terracini's lemma gives $h^1(\PP^{v-1},\Ii
_{(2u_V(o),X_V)
\cup (2u_V(o_{b-2}),X_V)}(1)) =0$ and hence $h^0(\PP^{v-1},\Ii _{(2u_V(o),X_V)\cup u_V(A_2)}(1)) < h^0(\PP^{v-1},\Ii _{(2u_V(o),X_V)\cup
u_V(A_f)}(1))$. Take $f\in V(-(2o)-A_2))$ with $f\notin V(-(2o)-A_1))$. Since $W$ is a local embedding at $o$, $o_1,\dots ,o_{b-2}$ are general and $\dim W \ge n+b-1$,
there is $g\in W(-2(o,X))$ such that $g(o_i)=0$ if and only if $i\ne b-2$. Use the image of $f\otimes g$.

\quad ({c}) In this step we prove that $\sharp (\Zz (X,b,q))=1$ for a general $q\in \tau (X,b)$. Fix a general $q\in \tau (X,b)$ and assume $\sharp (\Zz (X,b)) >1$ and so there are $Z,A\in \Zz (X,b)$
with $Z\ne A$. Since $\dim \tau (X,q) =b(n+1)-2$, a dimensional count shows that $\Zz (X,b,q)$ is finite for a general $q\in \tau (X,b)$. Hence we may assume
that $\Zz (X,b,q)$ is finite. A dimensional count gives that $Z$ and $A$ are general in $\Zz (X,b)$, but of course we do not assume any generality for $Z\cup A$. In particular
we may assume $Z\cap (I _V\cup I_W)=\emptyset$ and $A\cap (I_V\cup I_W) =\emptyset$.  Since $\dim V >b$, we get $\dim V(-Z) =\dim V -b >0$. Let $D\subset X$ be the hypersurface whose equation is a general 
element of $V(-Z)$. Let $E$ denote the residual scheme $\Res _D(Z\cup A)$ of $Z\cup A$ with respect to the effective Cartier
divisor
$D\subset X$. Since $Z\subset D$, we have $E =\Res _D(A)$. Thus $E$ is a closed subscheme of $A$ and $E=\emptyset$ if and only
if $A\subset D$. Note that each element of $\Zz (X,b)$ has only finitely many subschemes. Since
$\dim \tau (X,b) =b(n+1)-2$ and $q$ is general in $\tau (X,b)$, we have $q\notin \langle Z'\rangle$ for any $Z'\subsetneq Z$
and $q\notin \langle A'\rangle$ for any $A'\subsetneq A$. Since $q\in \langle Z\rangle \cap \langle A\rangle$, $A\ne Z$,
$q\notin \langle Z'\rangle$ for any $Z'\subsetneq Z$ and $q\notin \langle A'\rangle$ for any $A'\subsetneq A$, we have
$h^1(\PP^r,\Ii _{Z\cup A}(1)) >0$. 

Since $u_V(X)$ is not singular in codimension $1$ and it is embedded in a projective space of dimension $\ge n+2$, $u_V(X)$ is not tangentially degenerate (\cite[Theorem 3.1]{k}). By Lemma \ref{a2} applied to $X_V\subset \PP^{v-1}$ the scheme $u_V(Z)$ is the scheme-theoretic theoretical base locus of $X_V\cap \langle u_V(Z)\rangle$. Since $A\cap I_V=\emptyset$ and $A\ne Z$, $A$ is not contained in the base locus of $V(-Z)$. Since $D$ is a general element of $V(-Z)$, we get $A\nsubseteq D$, i.e. $E \ne \emptyset$.
Since $A$ is general in $\Zz (X,b)$ we have $\dim W(-A) =\dim W -\deg (A)$. Since $E\subseteq  A$, we have
$\dim W(-E) =w-\deg (E)$, a contradiction. The surjection $V\otimes W\to H^0(\Oo _{\PP^r}(1)_{|X})$ gives $h^1(\PP^r,\Ii
_{Z\cup A}(1)) =0$, a contradiction.

\quad (d)  In this step we prove that $r_X(q) > b$ for a general $q\in \tau (X,b)$. Since $\dim \tau (X,b)>\dim \sigma
_{b-1}(X)$ by step (b), we have $r_X(q)=b$. Take $Z\in \Zz (X,b,q)$ and $S\in \Ss (X,q)$. Since $S\in \Ss (X,q)$, there is no $S'\subsetneq S$ with $q\in \langle S'\rangle$. Since $\dim \tau (X,b)=b(n+1)-2$,
we may assume that $Z$ is general in $\Zz(X,b)$ and that $q\notin \langle Z'\rangle$ for any $Z'\subsetneq Z$. Since $Z$
is not reduced, we have $Z\ne S$. Hence
$h^1(\Ii _{Z\cup S}(1)) >0$. As in step
({c}) we see that
$Z$ is the intersection of the open set
$X\setminus I_V$ with the scheme-theoretic base locus of $V(-Z)$. Fix a general
$q\in
\tau (X,b)$ and assume
$r_X(q)
\le b$. Take
$Z\in
\Zz (X,b,q)$. If we have $S$ with $S\cap (I_V\cup I_W)=\emptyset$ and $\dim W(-S) =w-b$, then we may apply verbatim the
proof in step ({c}) with $S$ instead of $A$. If $S\cap (I_V\cup I_W)=\emptyset$ and $\dim V(-S) =v-b$, then we may apply
the proof in step ({c}) taking $(W,Z)$ instead of $(V,Z)$ and $(V,S)$ instead of $(W,A)$. Call $\tau\tau$ a
non-empty
open subset of $\tau (X,b)$ such that for each $q\in \tau \tau$ we have $r_X(q) =b$ and $q\in \langle Z\rangle$ with $Z$
sufficiently general in
$\Zz (X,b)$ (we need $\dim W(-Z) = w-b$, $\dim V(-Z) = v-b$, $Z\cap (I_V\cup I_W)=\emptyset$ and that $(X_V,u_V(Z))$ and $(X_W,u_W(Z))$ satisfy the thesis
of Lemma \ref{a2}). The set $\Ss (X,q)$ is constructible and hence it makes
sense to speak about the irreducible component of $\Ss (X,q)$ and of their dimension. Let $\sigma _b(X)'$ denote the set
of all $a\in \sigma _b(X)$ such that there is a finite set $B\subset X$ with $\sharp (B)=b$, $a\in \langle B\rangle$
and $a\notin \langle B'\rangle$ for any $B'\subsetneq B$. The set $\sigma _b(X)'$ is constructible (it is the image of a
an open subset of the abstract join of $b$ copies of $X$). Hence $\tau := \tau \tau \cap \sigma _b(X)'$ is constructible.
By assumption $\tau$ contains a non-empty open subset of $\tau (X,b)$ and hence it is irreducible
and of dimension $b(n+1)-2$. Let $\Gamma _b\subseteq X^{(b)}$ be the set of all $S\in \Ss (X,q)$ with $q\in \tau$.
Since $\dim \tau =n(b+1)-2$, we have $\dim \Gamma _b \ge nb-1$. If $\dim \Gamma _b = nb$, then for a general
$q\in \tau$ we may take as $S$ a general subset of $X$ with cardinality $b$, concluding the proof in this case.
Thus we may assume that each irreducible component of the constructible set $\Gamma _b$ has dimension $nb-1$.
Thus there is a non-empty open subset $\tau '$ of $\tau$ such that $\Ss (X,q)$ is finite for all $q\in \tau '$.
Restricting $\tau '$ if necessary we may assume that the positive integer $\sharp (\Ss (X,q))$ is the same for
all $q\in \tau '$. Let $X^{(b)}$ denote the symmetric product of $b$ copies of $X$ and let $m: X^b\to X^{(b)}$ be the quotient
map. Let $\Dd$ be an irreducible component of $m^{-1}(\Gamma _b)$ with dimension $nb-1$. For any $i=1,\dots ,b$ let
$\eta _i: X ^{b} \to X^{b-1}$ denote the projection onto the factors with indices in $\{1,\dots ,b\}\setminus \{i\}$. Since $\dim \Dd =nb-1$
either $\eta _i(\Dd)$ contains a non-empty open subset of $X^{b-1}$ or the closure of $\eta _i(\Dd)$ is a hypersurface
$\Delta$
of $X_{b-1}$ and $\Dd$ contains a non-empty open subset of $X\times \Delta$. Thus there is $j\in \{1,\dots ,s\}$
such that $\eta _{j|\Dd}$ is dominant. Thus for a general $q\in \tau$ we may find $S=\{p_1,\dots ,p_b\}\in \Ss (X,q)$
with $(p_1,\dots ,p_{b-1})$ general in $X^{b-1}$. Thus $\dim V(-S') = v-b+1$ and $\dim W(-S')=w-b+1$, where $S':= \{p_1,\dots
,p_{b-1}\}$. Hence $p_b$ is both in the base locus of $V(-S')$ and in the base locus of $W(-S)$. This is impossible, since
$X$ is embedded in $\PP^r$ and  (by our reduction at the beginning of the proof) the image of the map $\rho : V\otimes W\to
H^0(\Oo _{\PP^r}(1))$ is surjective.
\end{proof}

\begin{example}\label{i1}
Fix integer $s\ge 1$, $n_i>0$, $d_i$, $c_i$, $1\le i \le s$, such that $0 < c_i<d_i$ for all $i$.
Let $X\subset \PP^r$, $r+1=\prod _{i=1}^{s} \binom{n_i+d_i}{n_i}$, be the Segre-Veronese embedding of the multiprojective
space $\PP^{n_1}\times \cdots \times \PP^{n_s}$. Set $V:= H^0(\Oo _X(c_1,\dots ,c_s))$, $W:= H^0(\Oo _X(d_1-c_1,\dots
,d_s-c_s))$ and $n:= n_1+\cdots +n_s$. Fix an integer $b\ge 2$ such that $\prod _{i=1}^{s} \binom{n_i+c_i}{n_i} \ge b+n+2$,
$\prod _{i=1}^{s} \binom{n_i+d_i-c_i}{n_i}\ge n+b+2$
and either $s\ge 3$ or $s=2$ and $(c_1,c_2)\ne (1,1)$ or $s=1$ and $c_1\ge 3$. We claim that
$\dim
\tau (X,b) =b(n+1)-2$,
$\dim \sigma _b(X) = b(n+1)-1$ and $r_X(q)>b$ and $\sharp (\Zz (X,b,q)) =1$ for a general $q\in \tau (X,b)$.
By Remark \ref{zz0} to apply Theorem \ref{i1} it is sufficient to observe that the variety $\sigma _2(X_V)$ has dimension
$2n+1$, where
$X_V$ is the Segre-Veronese embedding of $X$ by the complete linear system $|\Oo _X(c_1,\dots ,c_s)|$, by \cite[Theorem
4.2]{ab}.
\end{example}

\section{Proof of Theorem \ref{i3}}

In this section $X=\PP^{n_1}\times \cdots \times \PP^{n_s}$. For any $i\in \{1,\dots ,s\}$ let $\pi _i: X\to \PP^{n_i}$ denote
the projection onto the $i$-th factor of $X$.
Set
$\Oo _X(\epsilon _i):=
\pi _i^\ast (\Oo _{\PP^{n_i}}(1))$. For any $E\subseteq \{1,\dots ,s\}$ set $\Oo _X(E):= \otimes _{i\in E}\Oo _X(\epsilon
_i)\in \mathrm{Pic}(X)$ and let $\pi _E: X\to \prod _{i\in E} \PP^{n_i}$ denote the projection onto the factors of $X$ with label in $E$.

\begin{proof}[Proof of Theorem \ref{i3}:] There are many papers, which could be used to see that $\dim \sigma _b(X) = b(n+1)-1$ (\cite{aop}, \cite{cgg}, and if $n_i=n$ for all $i$, \cite{li} (case $s=3$) and \cite{ah}, any $s$); this is also a consequence of \cite[Corollary 4.15]{bbcg}, which implies that $\Ss (X,o)$ is finite for a general $o\in \sigma _b(X)$. Take a general $q\in \tau (X,b)$. Thus there is $Z \in \Zz (X,b)$ with $q$ general in $\langle Z\rangle$. Since $q$ is general in $\tau (X,b)$, $Z$ is general in $\Zz (X,b)$ and $q$ is general in $\langle Z\rangle$.
In particular $q\notin \langle Z'\rangle$ for any $Z'\subsetneq Z$.
Take an integer $c\le b$ and assume the existence of $W\in (\Zz (X,c)\cup \Aa (X,c))$ with $q\in \langle W\rangle$ and $q\notin \langle W'\rangle$ for any $W'\subsetneq W$
and $W\ne Z$. Since $q\in (\langle Z\rangle \cap \langle W\rangle \setminus \langle Z\cap W\rangle)$, we have $h^1(\Ii _{Z\cup W}(1)) >0$.

\quad (a) Fix a decomposition $E\sqcup F$ with $\prod _{i\in E} (n_i+1) >b$ and $\prod _{i\in F} (n_i+1) >b$. In this step we prove that $c=b$ and that $h^1(\Ii _W(E)) =0$ and $H^0(\Ii _W(E))
=H^0(\Ii _Z(E))$; note that this would also imply that $\pi _{E|G}: G\to X_E$ is an embedding. Since $c\le b < \prod _{i\in E} (n_i+1)$, there is $D\in |\Oo _X(E)|$ with $D\supset W$. Thus $\Res _D(Z\cup W) = \Res _D(Z) \subseteq Z$. Since $\prod _{i\in E} (n_i+1) >b$ and $Z$ is general in $\Zz (X,b)$ we have $h^1(\Ii _Z(F)) =0$ and hence
$h^1(\Ii _{\Res _D(Z)}(F)) =0$. The residual exact sequence (\ref{eqt1}) of $\Ii _{Z\cup W}(1)$ with respect $D$ gives $h^1(D,\Ii _{(Z\cup W)\cap D,D}(E)) >0$.

\quad (a1) Assume $(Z\cup W)\cap D \ne Z\cup W$. Since $W\subset D$, we have $Z':= Z\cap D\subsetneq Z$. Since $\Res _D(W)=\emptyset$
and
$h^1(X,\Ii _{\Res _D(Z\cup W)}\otimes \Oo _X(F))=0$, the residual sequence of $\Ii _{Z\cup W}(1)$ with respect to $D$
gives $\langle Z\rangle \cap \langle W\rangle = \langle Z'\rangle \cap \langle W\rangle$. Thus $q\in \langle Z'\rangle$, a contradiction.

\quad (a2) Assume $(Z\cup W)\cap D =Z\cup W$, i.e. $Z\cup W\subset D$. By step (a1) we may assume that this is true for all
$D\in |\Ii _W(E)|$. Since $\deg (W)\le \deg (Z)$ and $h^1(\Ii _Z(E)) =0$, we get $\deg (W)=b$, $h^1(\Ii _W(E))=0$ (and in
particular
$\pi _{E|W}$ is an embedding) and that $H^0(\Ii _Z(E)) =H^0(\Ii _W(E))$.

\quad (a3) Exchanging  the role of $E$ and $F$ we also get $h^1(\Ii _W(F)) =0$, that $\pi _{F|W}$ is an embedding and that $H^0(\Ii _W(F)) =H^0(\Ii _Z(F))$. 

\quad (b) Take $E$ and $F$ as in step (a). Since $Z$ is general in $\Zz (X,b)$, the scheme $\pi _E(Z)$ is general in $\Zz (X_E,b)$. Since $H^0(\Ii _Z(E)) =H^0(\Ii _W(E))$,
Lemma \ref{a2} applied to $X_E$ gives $\pi _E(W)\subseteq \pi _E(Z)$. Since $\pi _{E|W}$ is an embedding, we first get $\pi _E(W) = \pi _E(Z)$ and
then that $W\in \Zz (X,q)$ and $W\notin \Aa (X,q)$. This is sufficient to see that $r_X(q) >b$. By step (a3) we also get
$\pi _F(W) =\pi _F(Z)$. Hence $\pi _i(W) = \pi _i(Z)$ for all $i\in \{1,\dots ,s\}$. This is not enough to say that $W$ has only finitely many possibilities
(obviously $W_{\red}$ has only finitely many possibilities) and so
to prove that $\dim \tau (X,b) =b(n+1)-2$ we need to work more. Fix again a general $q\in \tau (X,b)$ and assume that $\dim \tau (X,b) < b(n+1)-2$, i.e. assume
that $\Zz (X,b,q)$ is infinite. The set $\Zz (X,b,q)$ is constructible and hence it makes sense to speak about the dimensions of the irreducible components of $\Zz (X,b,q)$.
Since $\dim \tau (X,b) < b(n+1)-2$, each of the irreducible components of $\Zz (X,b,q)$ has positive dimension. Let $\Gamma$ be the irreducible component of $\Zz (X,b,q)$
containing $Z$.
A general $U\in \Gamma$ may be considered as a general element of $\Zz (X,b)$ and hence we may apply Lemma \ref{a2} for $X_V$ and $u_V(U)$ and for $X_W$ and $u_W(U)$.
Since there are only finitely many sets $W_{\red}$, $W\in \Gamma$, for a general $U\in \Gamma \setminus \{Z\}$ we have $U_{\red} = Z_{\red}$ and so $\deg (U\cap Z)=b-1$. Since $q\in (\langle Z\rangle
\cap \langle U\rangle \setminus \langle Z\cap U\rangle)$, we get $\langle U\rangle = \langle Z\rangle$ and hence $U\subset \langle Z\rangle$, contradicting Lemma \ref{a2}.
\end{proof}


\begin{thebibliography}{99}



\bibitem{ab} H. Abo and M. C. Brambilla,  On the dimensions of secant varieties of Segre-Veronese varieties, Ann. Mat. Pura Appl. (4) 192 (2013), no. 1, 61--92.

\bibitem{aop} H.~Abo, G.~Ottaviani, and C.~Peterson, 
Induction for secant varieties of Segre varieties, 
Trans. Amer. Math. Soc. 361 (2009), no. 2, 767--792.




 


\bibitem{a}  B. \r{A}dlandsvik, 
Joins and higher secant varieties, 
Math. Scand. 62 (1987),  213--222.

\bibitem{ah} T. Aladpoosh, and H. Haghighi, 
On the dimension of higher secant varieties of Segre varieties $\PP^n\times\dots\times\PP^n$,
J. Pure Appl. Algebra 215 (2011), no. 5, 1040--1052.


\bibitem{amr} C. Araujo, A. Massarenti, and R. Rischter, 
On nonsecant defectivity of Segre-Veronese varieties, 
Preprint arXiv:1611.01674 (2016).

\bibitem{acgh} E. Arbarello, M. Cornalba, P. Griffiths, and J. Harris, 
Geometry of algebraic curves, Vol. I,
Springer, Berlin - Heidelberg - New York, 1985.


\bibitem{b} E. Ballico, The $b$-secant variety of a smooth curve has a codimension $1$ locally closed subset whose points have rank at least $b+1$, arXiv: 1706.03633.


\bibitem{bb1}  E. Ballico and A. Bernardi, 
Tensor ranks on tangent developable of Segre varieties,  
Linear Multilinear Algebra 61 (2013) 881--894.

\bibitem{bbc} E. Ballico, A. Bernardi and L. Chiantini, On the dimension of contact loci and the identifiability of tensors,
arXiv: 1706.02746.


\bibitem{bbcg} E. Ballico, A. Bernardi, L. Chiantini and E. Guardo,  Bounds on the tensor rank, arXiv: 1705.02299.







\bibitem{bp} M. Bolognesi and G. Pirola, Osculating spaces and diophantine equations (with
an Appendix by P. Corvaja and U. Zannier), Math. Nachr. 284 (2011), 960--972.



\bibitem{bl} J. Buczy\'nski and J.M. Landsberg, 
On the third secant variety, 
J. Algebraic Combin. 40 (2014) 475--502.

\bibitem{cgg} M. V. Catalisano, A. V. Geramita, and A. Gimigliano, 
Secant varieties of $\PP^1\times \cdots \times \PP^1$ ($n$-times) are NOT defective for $n\ge 5$, 
J. Alg. Geom. 20 (2011), 295--327.





\bibitem{fh} W. Fulton and J. Hansen, A connectedness theorem for projective varieties, with applications to intersections and singularities of mappings, Ann. Math. 110 (1979),
no. 1, 159--166.




\bibitem{k} H. Kaji, On the tangentially degenerate curves, J. London Math. Soc. (2) 33 (1986), 430--440.

\bibitem{k1} H. Kaji, On the tangentially degenerate curves, II, Bull. Braz. Math. Soc., New Series 45(4), 745--752.





\bibitem{l} J.M. Landsberg, 
Tensors: Geometry and Applications, Graduate Studies in Mathematics, 
Amer. Math. Soc. Providence, 128  (2012). 

\bibitem{li} T. Lickteig, 
Typical tensorial rank.
 Linear Algebra Appl. 69 (1985), 95--120.


\bibitem{u} B. Ullery, On the normality of secant varieties, Adv. Math. 288 (2016), 631--647.


\end{thebibliography}
\end{document}